\documentclass[reqno]{cmsart}
\usepackage{amssymb, graphicx}
\usepackage{amsmath,amsthm,mathrsfs,amsfonts}
\usepackage{hyperref}
\usepackage{xcolor}
\usepackage{float}
\usepackage{pgf,tikz}
\usetikzlibrary{arrows}
\def \C {\mathbb{C}}
\def \D {\mathbb{D}}
\def \R {\mathbb{R}}
\def \T {\mathbb{T}}
\def \N {\mathbb{N}}

\def\d{\mathrm{d}}
\def\hol{\mathrm{Hol}}

\newtheorem{theorem}{Theorem}[section]
\newtheorem{lemma}[theorem]{Lemma}

\newtheorem{corollary}[theorem]{Corollary}

\begin{document}
	
	\title{A Note on the Phase Retrieval of Holomorphic Functions}

	\author[R. Perez III]{Rolando Perez III}
	
	\address{R. Perez III, Universit\'e de Bordeaux, CNRS, Bordeaux INP, IMB, UMR 5251,  F-33400, Talence, France}
	
	\address{Institute of Mathematics, University of the Philippines Diliman, 1101 Quezon City, Philippines}
	
	\email{roperez@u-bordeaux.fr}
	
	

	\begin{abstract}
		
		We prove that if $f$ and $g$ are holomorphic functions on an open connected domain, with the same moduli on two intersecting segments, then $f=g$ up to the multiplication of a unimodular constant,
		provided the segments make an angle that is an irrational multiple of $\pi$. We also prove that if $f$ and $g$ are functions in the Nevanlinna class, and if $|f|=|g|$ on the unit circle and on a circle inside the unit disc, then $f=g$ up to the multiplication of a unimodular constant.
		
		%We prove that if $f$ and $g$ are holomorphic functions on the unit disc, \bl{with the same moduli on two intersecting segments} inside the disc where the angle between these segments is an irrational multiple of $\pi$, then $f=g$ up to the multiplication of a unimodular constant. 
		%We prove that if $f$ and $g$ are holomorphic functions on an open connected domain, with the same moduli on two segments inside the domain, then $f=g$ up to the multiplication of a unimodular constant.
		
	\end{abstract}
	
	\maketitle 
	
	\section{Introduction}
	
	The study of phase retrieval involves the recovery of a function $f$ in some function space from given data about the magnitude of $f$ (phaseless information) and other assumptions on $f$, where these other assumptions may be in terms of some transform of $f$. Phase retrieval problems are widely studied because of their physical applications in fields of science and engineering. The most natural question asked on phase retrieval is about the uniqueness of the solution. However, phase retrieval problems generally have large solution sets, so additional assumptions related to the known data are usually added to reduce the solution set, or to consequently force the uniqueness of the solution. We refer the reader to the survey articles by Grohs et. al. \cite{GKR} and Klibanov et. al. \cite {Kli} for a more general perspective on the phase retrieval problem, together with some examples and further references.
	
	Phase retrieval problems have been formulated in both finite and infinite-dimensional cases. In turn, they have been solved using a diverse array of techniques, which include the use of tools from complex analysis. In some cases, the  problem shifts to the complex analytic scenario by holomorphic extensions or by integral transforms.
For instance, consider $f$ and $g$ to be two band-limited $L^2$ functions, then from the Paley-Wiener theorem, they are entire functions of finite type. Complex analysis was the key tool used by  Akutowicz  \cite{Ak1,Ak2}
(and  a  few  years  later  independently  by  Walther \cite{Wa} and Hofstetter \cite{Ho}) to determine all band-limited functions $g$ such that $|g|=|f|$. We now enumerate some further work
which used complex analytic tools. Grohs et. al. \cite{Gr} considered the recovery of a function in a modulation space from phaseless Gabor measurements, where they considered the short-time Fourier transform and used the Poisson-Jensen formula in their estimates. Waldspurger et. al. \cite{MW} solved a continuous case of the recovery of an $L^2$ function from the modulus of its wavelet transform by using Fourier transforms and holomorphic extensions to the upper half-plane. 
Moreover, McDonald \cite{Mc} has extended the work of Akutowicz to cover entire functions of finite genus.
The solutions were characterized with the help of Hadamard factorization. In \cite{Ja}, Jaming also used the Hadamard factorization to show that an entire function of finite order can be reconstructed from its modulus on two lines, where these lines intersect at an angle which is an irrational multiple of $\pi$. Bodmann et. al. \cite{Bod} then used this result with conformal mappings to show that a polynomial of degree at most $n-1$ can be determined by its magnitude at $4n-4$ well-chosen points in the complex plane. More information on phase retrieval problems from holomorphic measurements can be found in \cite{GKR}.
	
	Our aim in this paper is to generalize some uniqueness results on the phase retrieval of holomorphic functions on the unit disc. Indeed, in \cite{JKP1}, we considered the phase retrieval problem: given 
	$f\in L^2(\R)$ such that its Fourier transform satisfies an exponential decay condition, % with $\widehat f\in L^2(\R,e^{2c|x|}\,\mbox{d}x)$, 
	find all functions $g\in L^2(\R)$ such that $|f|=|g|$ on $\R$ and with its Fourier transform satisfying the same exponential decay condition as $f$. % Let us first sketch the solution of this problem. 
	Using a Paley-Wiener theorem and a conformal map,  the problem can be translated to the Hardy space on the unit disc. By the inner-outer factorization, the explicit form of the solution was obtained. One of our motivation stems from one of the coupled phase retrieval problems from \cite[Lemma 4.5]{JKP1}, which states that for $f$ and $g$ in the Hardy space on the unit disc with $|f|=|g|$ on $(-1,1)$ and on some segment inside the disc, $g$ can be obtained uniquely from $f$. For our first result, we extend this uniqueness result to holomorphic functions on open connected domains.
	
	On the other hand, our next objective is to improve the result of  Boche et. al. \cite[Theorem 3]{Bo}, which states that given functions $f$ and $g$ in the Hardy space on the disc without singular inner part, if $|f|=|g|$ on the unit circle and on a smaller circle inside the unit circle, then $g$ can be uniquely determined from $f$. Moreover, they also showed by construction that the Blaschke product associated with $g$ can be uniquely recovered by its modulus on a smaller circle inside the unit circle. For our other result, we extend this uniqueness result to all functions in the Nevanlinna class, regardless of the presence of the singular parts.
	
	This work is organized as follows. Section 2 includes a quick review of definitions and results related to spaces of holomorphic functions on the disc, and the statements of our results. Section 3 is devoted to the proofs of our results.

\section{Preliminaries and Statement of Results}

\subsection{The Nevanlinna Class and the Hardy Space on the Disc}

Let $\mathbb D$ be the unit disc and $\T:=\partial\D$ be its boundary. We denote by $D(a,r)$ the disc centered at $a\in\C$ with radius $r>0$, $r\D=D(0,r)$ and $r\T=\partial D(0,r)$. %We now recall some classical spaces of holomorphic functions on $\D$.
We denote by $H^{\infty}(\D)$ the space of bounded holomorphic functions on $\D$. The Nevanlinna class is defined by
$$
\mathcal N=\left\{ \varphi\in\hol(\D): \varphi=\dfrac{f}{g},~f,g\in H^{\infty}(\D) \right\}.
$$
Note that the radial limit given by $
\varphi^*(\zeta)=\lim_{r\rightarrow 1} \varphi(r\zeta)$ exists almost everywhere in $\T$ and $\log|\varphi^*|\in L^1(\T)$. For $\varphi\in\mathcal N$, $\varphi$ has a factorization \cite[Theorem 2.9]{Du} of the form
\begin{equation}
\label{eq:nevfact}
\varphi=\dfrac{e^{i\gamma}B_\varphi S_{\nu_1}O_\varphi}{S_{\nu_2}}
\end{equation}
where $e^{i\gamma}\in\T$, $B_\varphi$ is the Blaschke product formed from the zeros of $\varphi$, $S_{\nu_1}$ and $S_{\nu_2}$ are singular inner functions, and $O_\varphi$ is the outer part of $\varphi$. Here, the Blaschke product is defined for all $z\in\mathbb D$ as 
\begin{equation*}
\label{eq:blaschke}
B_\varphi (z)=z^k\prod_{\alpha\in \Lambda}\dfrac{\alpha}{|\alpha|}\dfrac{\alpha-z}{1-\bar{\alpha}z}
\end{equation*}
where $\Lambda$ is the set of nonzero zeros of $\varphi$ counted with multiplicity, which satisfy the Blaschke condition $\sum_{\alpha\in\Lambda} (1-|\alpha|)<\infty$.
The singular inner function is given by
\begin{equation*}
\label{eq:singinner}
S_\nu(z)=\exp\left(\int_{\mathbb T}\dfrac{z+\zeta}{z-\zeta}~\d\nu(\zeta)\right),
\end{equation*}
where $\nu$ is a finite positive singular measure on $\T$. Finally, the outer part of $\varphi$ is given by
\begin{equation*}
\label{eq:outer} O_\varphi(z)=\exp\left(\dfrac{1}{2\pi}\int_{\T}\dfrac{\zeta+z}{\zeta-z}\log|\varphi^*(\zeta)|~|\d\zeta|\right).
\end{equation*}
It is easy to see that for $f\in\mathcal{N}$, if $f(z_0)=0$ then $\dfrac{f(z)}{z-z_0}\in\mathcal{N}$.
We also recall the subclass of $\mathcal N$ called the Smirnov class, defined by
$$
\mathcal N^+=\left\{ \varphi\in\hol(\D): \varphi=\dfrac{f}{g},~f,g\in H^{\infty}(\D),\,g\text{ is outer}\right\}.
$$ 
Furthermore, the Generalized Maximum Principle \cite[Section 3.3.1, (g)]{Ni} states that if $\varphi\in\mathcal N^+$ with radial limit $\varphi^*\in L^p(\T)$, then $\varphi$ belongs to classical Hardy space on the disc $H^p(\D)$, $1\leq p\leq \infty$.
Note that every function in $H^p(\D)$ has a factorization as in \eqref{eq:nevfact} with $S_{\nu_2}=1$, and this factorization is unique.

\subsection{Statement of Results} 

We begin with a simple observation. Let $\omega, \Omega$ be open connected
sets such that $\omega\subset \Omega$, let  $f,g\in\hol(\Omega)$ and suppose that $|f|=|g|$ on $\omega$.
Then,  for some $c\in \T$, $g =cf$ on $\Omega$.
Indeed, we can assume that  $|f|=|g|$ on a closed disc
$\overline{D}$, hence $f$ and $g$ have the same zeros with the same
multiplicities on $\overline{D}$. Consequently, $F=f/g$ is a
holomorphic function on $D$ and $|F|=1$. Therefore $0=\Delta
|F^2|=|F'|^2$ on $D$ and hence $F=c$ for some $c\in \T$.

In the same spirit when $\omega$ is not open, we have shown in the paper \cite[Lemma 4.5]{JKP1} that if $f,g\in H^2(\D)$ and 
$$
|f(z)|=|g(z)|,\qquad z\in(-1,1)\cup e^{i\alpha}(-1,1)
$$
where $\alpha\notin \pi\mathbb Q$, then $f$ and $g$ are equal up to the multiplication of a unimodular constant. For this result, uniqueness was established by showing that the Blaschke products, singular inner parts, and outer parts of $f$ and $g$ are equal. Our first result consists in showing that this is true for arbitrary holomorphic functions in an open connected domain.
\begin{theorem}
	\label{thm:gen-omega}
	Let $\Omega$ be an open connected domain.  Let $f,g\in\hol(\Omega)$ and suppose that
	\begin{equation}
	\label{eq:gen-omega1}
	|f(z)|=|g(z)|,\qquad z\in I\cup I_\alpha,
	\end{equation}
	where $I$ and $I_\alpha$ are segments inside $\Omega$, $I_\alpha$ is the $\alpha$-rotation of $I$ about the midpoint of $I$, and $\alpha\notin \pi\mathbb Q$. Then $g(z)=cf(z)$ for all $z\in\Omega$ and for some $c\in\T$.
\end{theorem}

%As an immediate corollary, we get the following (probably well known) fact:
%
%\begin{corollary}
%\label{cor:triv}
%Let $\omega\subset\Omega$ be open sets. Let $f,g\in\hol(\Omega)$ and suppose that $|f|=|g|$
%on $\omega$. $g(z)=cf(z)$ for all $z\in\Omega$ and for some $c\in\T$.
%\end{corollary}
%
%Indeed, note that $\omega$ contains two segments of the form appearing in Theorem \ref{thm:gen-omega}, thus $g=cf$ on $\omega$ and thus on $\Omega$ by unique continuation of entire functions.

Let us now see what is happening if segments are replaced by circles. To do so, recall that if $f$ and $g$ are outer functions in $\mathcal N$ such that $|f|=|g|$ almost everywhere on $\T$, then $f$ is equal to $g$ up to the multiplication of a unimodular constant. Now, Boche et. al. \cite[Theorem 3]{Bo} solved a more general problem: if $f,g\in H^1(\D)$ have no singular parts (i.e. $f=B_fO_f$, $g=B_gO_g$) and $|f|=|g|$ almost everywhere on $\T$ and $|f|=|g|$ on $\rho\T$ for some $0<\rho<1$, then $g$ is uniquely determined by $f$. The heart of their proof is the explicit construction of the Blaschke product associated to $g$, as the equality of the outer parts immediately follow. For our next result, with the same equalities of the moduli on the aformentioned circles, we improve the result by Boche et. al. by showing that uniqueness holds for all functions in $\mathcal N$. We emphasize that in this result, we may either have the presence or the absence of the singular inner part. 

\begin{theorem}
	\label{thm:2circ}
	Let $f,g\in\mathcal N$ and let $\rho\in(0,1)$. If 
	\begin{equation}
	\label{eq:2circ}
	|f(\zeta)|=|g(\zeta)|,\quad\text{a.e. }\zeta\in\T\quad\text{and}\quad |f(z)|=|g(z)|,\quad z\in\rho\T
	\end{equation}
	then $g(z)=cf(z)$ for all $z\in\D$ and for some $c\in\T$.
\end{theorem}

\section{Proofs and Remarks}

In this section, we present the proofs of our  results, and some immediate consequences of them. %\bl{It will be convenient to denote the $r$-dilation of a function $f$ by $f_r(z)=f(rz)$ for some $r>0$.}

\subsection{Proof of Theorem \ref{thm:gen-omega}}
Observe that replacing $f(z)$ by $f(z_0+rze^{i\beta})$ with $z_0,r,\beta$ appropriately chosen,
we may assume that

-- $(1+\varepsilon)\overline{\D}\subset\Omega$ for $\varepsilon>0$,

-- $I=(-1,1)$, $I_\alpha=e^{i\alpha}(-1,1)$.

Note that now $f,g\in H^2(\D)$ so that we could apply \cite[Lemma 4.5]{JKP1} and obtain $g=cf$, for some $c\in \T$.
We will give an alternative simpler proof.

Note that, as the zeros of $f$ and $g$ are isolated, by choosing $r$ small enough, we can assume that
they have at most one zero in  $(1+\varepsilon)\overline{\D}$ which is at $0$. We can write
\begin{equation}
\label{eq:gen-omegax}
f(z)=z^ke^{\varphi(z)}\qquad\text{and}\qquad g(z)=z^l e^{\psi(z)},\qquad  z\in \overline{\D}
\end{equation}
where $\varphi,\psi\in \hol((1+\varepsilon)\D)$ and $k,l$ are nonnegative integers.
As $|f(x)|=|g(x)|$ for $x\in(-1,1)$ we conclude that $k=l$.

It remains to show that the zero-free factors of $f$ and $g$ are equal up to a unimodular constant. First, we note that \eqref{eq:gen-omega1} is equivalent to 
\begin{equation*}
\text{Re}\,\varphi(t)=\text{Re}\,\psi(t)\qquad\text{and}\qquad\text{Re}\,\varphi(te^{i\alpha})=\text{Re}\,\psi(te^{i\alpha}),\qquad t\in(-1,1).
\end{equation*}
Since $\varphi,\psi\in\hol(\D)\cap C^\infty(\overline{\D})$, $\text{Re}\,\varphi$ and $\text{Re}\,\psi$ are harmonic and
\begin{equation}
\label{eq:phipsi}
\varphi(z)=\sum_{n\ge 0} |z|^n\widehat{\varphi}(n)e^{in\theta}
\qquad\text{and}\qquad
\psi(z)=\sum_{n\ge 0} |z|^n\widehat{\psi}(n)e^{in\theta}
\end{equation}
for $z=|z|e^{i\theta}\in\D$. It follows that
\begin{align*}
\text{Re}\,\varphi(z)&=\text{Re}\,\widehat{\varphi}(0)+\sum_{n\in\N} |z|^n\dfrac{\widehat{\varphi}(n)e^{in\theta}+\overline{\widehat{\varphi}(n)}e^{-in\theta}}{2}
\intertext{and}
\text{Re}\,\psi(z)&=\text{Re}\,\widehat{\psi}(0)+\sum_{n\in\N} |z|^n\dfrac{\widehat{\psi}(n)e^{in\theta}+\overline{\widehat{\psi}(n)}e^{-in\theta}}{2}.
\end{align*}
Thus \eqref{eq:gen-omega1} and \eqref{eq:gen-omegax} together with the sums above imply that
$\text{Re}\,\varphi(t)=\text{Re}\,\psi(t)$  for $t\in(-1,1)$, if and only if 
$$
\text{Re}\,\widehat{\varphi}(n)=\text{Re}\,\widehat{\psi}(n),\qquad n\in\N,
$$
and $\text{Re}\,\varphi(te^{i\alpha})=\text{Re}\,\psi(te^{i\alpha})$ if and only if 
$$
\text{Re}\,\widehat{\varphi}(0)=\text{Re}\widehat{\psi}(0)
\qquad\text{and}\qquad
\text{Re}\,(\widehat{\varphi}(n)e^{in\alpha})=\text{Re}\,(\widehat{\psi}(n)e^{in\alpha})
$$
for all $n\in\N$. In other words
\begin{align}
\label{eq:rho}
\begin{split}
\begin{cases}
\text{Re}\,\left(\widehat{\varphi}(n)-\widehat{\psi}(n)\right)=\left< \widehat{\varphi}(n)-\widehat{\psi}(n),1\right>_\C=0,                                      \\[2mm]
\text{Re}\,\left(\left(\widehat{\varphi}(n)-\widehat{\psi}(n)\right)e^{in\alpha}\right)
=\left<\widehat{\varphi}(n)-\widehat{\psi}(n) ,e^{-in\alpha}\right>_\C=0,
\end{cases}
\end{split}
\end{align}
for all $n\in\N$. Since $\alpha\notin\pi\mathbb Q$, $\{1,e^{-in\alpha}\}$ is a basis for $\C$ when $n\not=0$
so by \eqref{eq:rho}, we have $\widehat{\varphi}(n)=\widehat{\psi}(n)$.
On the other hand, as $\text{Re}\,\widehat{\varphi}(0)=\text{Re}\,\widehat{\psi}(0)$ there exists
$\lambda\in\R$ such that $\widehat{\psi}(0)=\widehat{\varphi}(0)+i\lambda$. It follows from \eqref{eq:phipsi} that
$\psi=\varphi+i\lambda$ thus $g(z)=e^{i\lambda}f(z)$ for all $z\in\D\subset\Omega$. As $\Omega$ is connected and
$f,g\in\hol(\Omega)$, this implies that $g(z)=e^{i\lambda}f(z)$ also holds on $\Omega$.

\subsection{Proof of Theorem \ref{thm:2circ}}

%Since $|f(\zeta)|=|g(\zeta)|$ for almost every $\zeta\in\T$, then the outer parts of $f$ and $g$ are equal up to the multiplication of a unimodular constant. Hence it suffices to consider the case when $f$ and $g$ are inner.

We begin with a simple observation. If $z_0\in\rho\T$ is a zero of $f$, we write $f(z)=(z-z_0)^k\tilde f(z)$
and $g(z)=(z-z_0)^j\tilde g(z)$ with nonnegative integers $j,k$ and $\tilde f,\tilde g\in\mathcal{N}$ nonvanishing at $z_0$.
Then $|f|=|g|$ on $\rho\T$ reads
$$
|(z-z_0)^k\tilde f(z)|=|(z-z_0)^j\tilde g(z)|,\qquad z\in\rho\T
$$
and this implies that $k=j$. Therefore, $f$ and $g$ have the same zeros on $\rho\T$ with the same multiplicities. We may thus write $f=Pf_1$ and $g=Pg_1$ with $P$ a polynomial which has all the zeros in $\rho\T$ and $f_1,g_1\in\mathcal{N}$ nonvanishing on $\rho\T$. Then $|f_1|=|g_1|$ on $\rho\T\cup\T$. In other words, we may assume that $f$ and $g$ do not vanish on $\rho\T$. 

%With this, it would suffice to show the result when $f$ and $g$ have no zeros on $\rho\T$.

Let $\{a_1,\dots,a_n\} $ and $\{b_1,\dots,b_m\}$ be the zeros of $f$ and $g$ on $\rho\D$ respectively, counted with multiplicities. For all $z\in\C$, write 
$$
P_f(z)=\prod_{i=1}^n\frac{\rho(z-a_i)}{\rho^2-\bar a_iz}\qquad\text{and}\qquad P_g(z)=\prod_{i=1}^m\frac{\rho(z-b_i)}{\rho^2-\bar b_iz}.
$$
Notice first that if $z\in\rho\T$, $|P_f(z)|=|P_g(z)|=1$. Further $\dfrac{f}{P_f}$ and $\dfrac{g}{P_g}$ do not vanish in $\rho\D$. By the Poisson-Jensen formula, for all $z\in\rho\D$ we have
\begin{align}
\label{eq:fp-gp}
\begin{split}
\log\bigg\vert\dfrac{f(z)}{P_f(z)}\bigg\vert&=\dfrac{1}{2\pi}\int_{-\pi}^{\pi}\text{Re}\left(\dfrac{\rho e^{i\theta}+z}{\rho e^{i\theta}-z}\right)\log|f(\rho e^{i\theta})|\,\d\theta\\
&=\dfrac{1}{2\pi}\int_{-\pi}^{\pi}\text{Re}\left(\dfrac{\rho e^{i\theta}+z}{\rho e^{i\theta}-z}\right)\log|g(\rho e^{i\theta})|\,\d\theta\\
&=\log\bigg\vert\dfrac{g(z)}{P_g(z)}\bigg\vert.
\end{split}
\end{align}
But then $\bigg\vert\dfrac{g(z)}{P_g(z)}\bigg\vert=\bigg\vert\dfrac{f(z)}{P_f(z)}\bigg\vert$ for all $z\in\rho\D$. Thus we get that there is some $\tilde c\in\T$ such that
$\dfrac{g(z)}{P_g(z)}=\tilde c\dfrac{f(z)}{P_f(z)}$ for all $z\in\rho\D$. Finally, as an identity between
holomorphic functions on $\D$, this is valid for all $z\in\D$. In particular, taking $z=re^{i\theta}$ and 
$r\longrightarrow 1$, we get $\dfrac{g(z)}{P_g(z)}=\tilde c\dfrac{f(z)}{P_f(z)}$ for almost every $z\in\T$. 
By \eqref{eq:2circ}, we have $\bigg\vert\dfrac{1}{P_f}\bigg\vert=\bigg\vert\dfrac{1}{P_g}\bigg\vert$ on $\T$ as well. Hence, by the following lemma, we get $\dfrac{1}{P_f}=\dfrac{c}{P_g}$ for some $c\in\T$, which implies $g=c\tilde cf$.

\begin{lemma}
	\label{lem:mero}
	Let $F$ and $G$ be meromorphic on $\C$ without poles on $\T\cup\rho\T$ for $0<\rho\ne1$. Suppose
	\begin{equation}
	\label{eq:mero-1}
	|F(\zeta)|=|G(\zeta)|\qquad\text{and}\qquad|F(\rho\zeta)|=|G(\rho\zeta)|,\qquad\zeta\in\T.
	\end{equation} 
	Then $F$ and $G$ have the same zeros and poles in $\C\setminus\{0\}$, with the same multiplicities. In particular, if $F$ and $G$ are rational functions that satisfy \eqref{eq:mero-1}, then $G=cF$ outside the poles with $c\in\T$.
\end{lemma}

\begin{proof}
	
	Let $F$ be meromorphic on $\C$ and $z_0\in\C$. Write  $F(z)=(z-z_0)^k\tilde F(z)$ and for some $k\in\mathbb Z$ with $\tilde F(z_0)\ne 0$. Define the multiplicity $k:=m_F(z_0)$, where $k>0$ if $z_0$ is a zero of $F$, and $k<0$ if $z_0$ is a pole of $F$. In particular, for meromorphic functions $F_1,F_2$ on $\C$, we have $m_{F_1F_2}=m_{F_1}+m_{F_2}$, $m_{F_1/F_2}=m_{F_1}-m_{F_2}$ and if $F_1=F_2$ in the neighborhood of a pole $z_0$, then $m_{F_1}=m_{F_2}$.
	
	First, note that if $\rho>1$, we replace $F$ and $G$ by $f(z/\rho)$ and $g(z/\rho)$, and replace $\rho$ by $1/\rho<1$. We thus assume that $\rho<1$. Observe that \eqref{eq:mero-1} is equivalent to
	$$
	F(z)\overline{F\left(\dfrac{1}{\bar z}\right)}=	G(z)\overline{G\left(\dfrac{1}{\bar z}\right)},\qquad z\in\T
	$$
	and
	$$
	F(z)\overline{F\left(\dfrac{\rho^2}{\bar z}\right)}=G(z)\overline{G\left(\dfrac{\rho^2}{\bar z}\right)},\qquad z\in\rho\T.
	$$
	As an identity between meromorphic functions in $\C$, these equations are also valid for $z\in\C$ not a pole of any of the functions involved. As poles are isolated, we have for $z\not=0$,
	\begin{equation}
	\label{eq:mero-proof-1}
	m_F(z)+m_F\left(\dfrac{1}{\bar z}\right)=m_G(z)+m_G\left(\dfrac{1}{\bar z}\right)
	\end{equation}
	and
	\begin{equation}
	\label{eq:mero-proof-2}
	m_F(z)+m_F\left(\dfrac{\rho^2}{\bar z}\right)=m_G(z)+m_G\left(\dfrac{\rho^2}{\bar z}\right).
	\end{equation}
	Now, \eqref{eq:mero-proof-1} gives
	$$
	m_F(z)-m_G(z)=m_G\left(\dfrac{1}{\bar z}\right)-m_F\left(\dfrac{1}{\bar z}\right)=m_F(\rho^2z)-m_G(\rho^2z)
	$$
	with \eqref{eq:mero-proof-2} applied to $1/\bar z$, for $z\in\C\setminus\{0\}$. If $z_0\ne 0$ is such that
	$m_F(z_0)\not=m_G(z_0)$ then
	\begin{align*}
	0&\not=m_F(z_0)-m_G(z_0)=m_F(\rho^2z_0)-m_G(\rho^2z_0)=\cdots\\
	&=m_F(\rho^{2k}z_0)-m_G(\rho^{2k}z_0)=m_{F/G}(\rho^{2k}z_0)
	\end{align*}
	for all $k\in\N$. But then $F/G$ is meromorphic and either has $\rho^{2k}z_0$ as a zero ($m_{F/G}(\rho^{2k}z_0)>0$) or as a pole ($m_{F/G}(\rho^{2k}z_0)<0$) for every $k$. Letting $k\longrightarrow\infty$ we have $\rho^{2k}z_0\longrightarrow 0$. As $z_0\ne 0$, this contradicts the fact that zeros and poles of $F/G$ are isolated. Hence, $F$ and $G$ have the same nonzero zeros and poles with the same multiplicities.	
	
	Furthermore, if $F$ and $G$ are rational functions, then they have same zeros and poles in $\C\setminus\{0\}$,
	thus there exists $c\in\T$ and $m\in\mathbb Z$ such that $G=cz^mF$. But then
	\eqref{eq:mero-1} implies that 
	
	--- on one hand $|F(\rho e^{it})|=|c|\rho^m|F(\rho e^{it})|$ for all $t$ thus $|c|\rho^m=1$.
	
	--- on the other hand $|F(e^{it})|=|c||F(e^{it})|$ thus $|c|=1$.
	
	As $\rho<1$ this implies $m=0$ and then $G=cF$.
\end{proof}

\begin{corollary}
	Let $F,G$ be two meromorphic functions in $\C$ with no pole at $0$ such that 
	$$
	|F(\zeta)|=|G(\zeta)|\qquad\text{and}\qquad|F(\rho\zeta)|=|G(\rho\zeta)|,\qquad\zeta\in\T,\,0<\rho\not=1.
	$$
	Then there exists $c\in\mathbb{T}$ such that $G=cF$.
\end{corollary}

\begin{proof}
	As seen in the previous proof, we can assume that $\rho<1$.
	
	The previous lemma then shows that $F$ and $G$ have the same nonzero poles in $\C$. Note that there is at most a finite number of such poles in a neighborhood of $\D$. We can factor them out and write
	$F=\tilde F/P$, $G=\tilde G/P$ with $\tilde F,\tilde G$ having no pole in a neighborhood of $\D$
	and $P$ a polynomial. But then $\tilde F,\tilde G\in\mathcal{N}$ and $|\tilde F|=|\tilde G|$ on $\T\cup\rho\T$.
	The previous theorem shows that there exists $c\in\T$ such that $\tilde G=c\tilde F$ thus $G=cF$ on $\D$ and thus on $\C$.

\end{proof}

\subsection*{Acknowledgements}
The author is supported by the CHED-PhilFrance scholarship from Campus France and the Commission of Higher Education (CHED), Philippines. The author also wants to thank his supervisors in the Universit\'e de Bordeaux, Philippe Jaming and Karim Kellay, for their very helpful comments and suggestions.

\end{document}